\documentclass[12pt]{amsart}
\usepackage{fourier}
\usepackage{amssymb}
\usepackage{amsmath}
\usepackage{amscd}
\usepackage{amsthm}
\usepackage[centertags]{amsmath}
\usepackage{amsfonts}
\usepackage{newlfont}
\usepackage{graphicx}
\usepackage[usenames]{color}
\usepackage{amsfonts, amssymb}
\usepackage{mathrsfs}
\usepackage{latexsym}
\usepackage{tikz}
\usepackage{verbatim}
\usepackage{tabulary,tabularx}
\usepackage[all]{xy}

\usepackage[latin1]{inputenc}

\newtheorem{theorem}{Theorem}[section]

\newtheorem{lemma}[theorem]{Lemma}

\theoremstyle{definition}
\newtheorem{remark}[theorem]{Remark}

\theoremstyle{definition}

\theoremstyle{remark}

\newcommand{\dis}{\displaystyle}

\newcommand{\Jj}{\mathcal J}

\newcommand{\PP}{\mathcal P}

\newcommand{\NN}{\mathbb N}
\newcommand{\zN}{\mathbb N}

\newcommand{\CC}{\mathbb C}
\newcommand{\DD}{\mathbb D}
\newcommand{\RR}{\mathbb R}
\newcommand{\jj}{\mathbf{j}}
\newcommand{\ii}{\mathbf{i}}

\usepackage{mathtools}

\DeclareMathOperator{\mon}{mon}

\newcommand{\jbn}[1]{\mathcal J(#1,n)}

\newcommand{\bj}{\mathbf j}

\newcommand{\Pp}{\mathcal{P}}
\newcommand{\Rr}{\mathcal{R}}

\newcommand{\dom}{\mbox{dom}}

\begin{document}
	\title{Mixed Bohr radius  in several variables}
	
	\author{Daniel Galicer \and Mart\'in Mansilla \and Santiago Muro}
	\thanks{This work was partially supported by projects CONICET PIP 1122013010032,   ANPCyT PICT 2015-2224,  ANPCyT PICT 2015-2299, % 2011-1456,  ANPCyT PICT 11-0738
		UBACyT 20020130300057BA, UBACyT20020130300052BA,UBACyT 20020130100474B
		. The second author was supported by a CONICET doctoral fellowship.}

	\address{ Departamento de Matem\'{a}tica - Pab I,
		Facultad de Cs. Exactas y Naturales, Universidad de Buenos Aires
		(1428) Buenos Aires, Argentina and CONICET} \email{dgalicer@dm.uba.ar} \email{mmansilla@dm.uba.ar}
	\email{smuro@dm.uba.ar}
	\keywords{Bohr radius, power series, homogeneous polynomials, unconditional bases, domains of convergence for monomial expansions}
	\subjclass[2010]{32A05,32A22 (primary),46B15,46B20,46G25,46E50 (secondary)}
	
%	\keywords{Hardy-Littlewood inequalities, unimodular polynomials, unconditionality in spaces of polynomials, multivariable von Neumann's inequality}
%	\subjclass[2010]{46G25,15A60,47H60,11C08,15A69,47A30}
	
	\begin{abstract}
		Let $K(B_{\ell_p^n},B_{\ell_q^n}) $ be the $n$-dimensional $(p,q)$-Bohr radius for holomorphic functions on $\CC^n$. That is, $K(B_{\ell_p^n},B_{\ell_q^n}) $ denotes the greatest constant $r\geq 0$ such that for every entire function $f(z)=\sum_{\alpha} c_{\alpha} z^{\alpha}$ in $n$-complex variables, we have the following (mixed) Bohr-type  inequality
		$$\sup_{z \in  r \cdot B_{\ell_q^n}} \sum_{\alpha} \vert c_{\alpha} z^{\alpha} \vert \leq \sup_{z \in  B_{\ell_p^n}} \vert f(z) \vert,$$
		where  $B_{\ell_r^n}$ denotes the closed unit ball of the $n$-dimensional sequence space $\ell_r^n$.
		
		For every $1 \leq p, q \leq \infty$, we exhibit the exact asymptotic growth of the $(p,q)$-Bohr radius as $n$ (the number of variables) goes to infinity. 
	\end{abstract}

	\maketitle

	\section{Introduction}

 At the early twentieth century, during the course of his investigations on the famous  Riemann $\zeta$ function, Harald Bohr \cite{bohr1913ueber,bohr1914theorem} devoted great efforts in the study and development of a general theory of Dirichlet series. A Dirichlet series is just an expression  of the form
 $$D(s) = \sum_{n \geq 1} \frac{a_n}{n^s},$$
 where $a_n \in \CC$ and $s= \sigma + it$ is a complex variable. The regions of convergence, absolute convergence and uniform convergence of these series define half-planes of the form $[ Re(s) > \sigma_0 ]$ in the complex field. Bohr was mainly interested in controlling the region of convergence of a series. To achieve this, he related different types of convergence and focused on finding the width of the greatest strip for which a Dirichlet series can converge uniformly but not absolutely. This question is popular and known nowadays as the \emph{Bohr's absolute convergence problem}.

 Although the solution of this problem problem appeared two decades after it was proposed (given by Bohnenblust and Hille \cite{bohnenblust1931absolute} who showed  that the maximum width of this strip is $\frac{1}{2}$), Bohr made major contributions in the area (arguably, even more important than the solution of the problem itself) in order to tackle it. He discovered a deep connection among Dirichlet series and power series in infinitely many variables. Given a Dirichlet series $D(s)=\sum_{n \geq 1} \frac{a_n}{n^s}$, he considered for each $n \in \NN$ the prime decomposition $n = p_1^{\alpha_1} \dots p_r^{\alpha_r}$ (where $(p_k)_{k \in \NN}$ denote the sequence formed by the ordered primes) and defined $z= (p_1^{-s}, \dots, p_r^{-s})$. Thus,
 $$D(s)=\sum_{n \geq 1} a_n (p_1^{-s})^{\alpha_1}  \dots (p_r^{-s})^{\alpha_r} = \sum a_n z_1^{\alpha_1} \dots z_r^{\alpha_r}.$$

 This correspondence, known as the \emph{Bohr transform} is not just formal: it gives an isometry between suitable spaces of Dirichlet series and power series \cite{hedenmalm1997hilbert}. The Bohr transform allows to transform/translate problems about Dirichlet series in terms of power series and tackle them with complex analysis techniques.
 This cycle of ideas brought Bohr to ask whether is possible to compare the absolute value of a power series in one complex variable with the sum of the absolute value of its coefficients. He manged to prove the following result nowadays referred as \emph{Bohr's inequality}:

{\sl The radius $r=\frac{1}{3}$ is the largest value for which the following inequality holds:

\begin{equation}
\sum_{n \geq 0} |a_n| r^n \leq \sup_{z \in \DD} \vert \sum_{n \geq 0} a_n z^n \vert,
\end{equation}
for every entire function $f(z)= \sum_{n \geq 0} a_n z^n$ on the unit disk $\DD$ such that $\sup_{z \in \DD} \vert f(z) \vert < \infty$.

}

As a matter of fact, Bohr's paper \cite{bohr1914theorem}, compiled by G. H. Hardy from correspondence, indicates that
Bohr initially obtained the radius $\frac{1}{6}$, but this was quickly improved to the sharp
result by M. Riesz, I. Schur, and N. Wiener, independently. Bohr's article presents
both his own proof and the one of his colleagues.

 This interesting inequality was overlooked during many years until the end of the twentieth century. In particular, Dineen and Timoney \cite{dineen1989absolute}, Dixon \cite{dixon1995banach}, Boas and Khavinson \cite{khavinson1997bohr}, Aizenberg \cite{aizenberg2000multidimensional} and Boas \cite{boas2000majorant} retook this work and use it in different contexts and/or generalize it. Several of these authors analyzed if a similar phenomenon occurs for power series in many variables. For each Reinhardt domain $\mathcal R$, they introduced the notion of the \emph{Bohr radius} $K(\mathcal R)$ as the biggest $r\geq 0$ such that for every analytic function $f(z)=\sum_{\alpha} a_{\alpha} z^{\alpha}$ bounded on $\mathcal R$, it holds:

\begin{equation}
\sup_{z \in r \cdot \mathcal R } \sum_{\alpha} \vert a_{\alpha} z^{\alpha} \vert \leq \sup_{z \in \mathcal R} \vert f(z) \vert.
\end{equation}

Note that with this notation, Bohr's inequality can be formulated simply as $K(\DD) = \frac{1}{3}$. Surprisingly, the exact value of the Bohr radius is unknown for any other domain. The central results of \cite{khavinson1997bohr,boas2000majorant} contained a (partial) successful estimate for the Bohr radius for the complex unit  balls  of $\ell_p^n$, $1\leq p \leq \infty$.

The gap between the upper and lower estimates in this papers leaded many efforts to compute the exact asymptotic order of $K(B_{\ell_p^n})$, for $1 \leq p \leq \infty$.

To obtain the upper bounds Boas \cite{boas2000majorant} generalized in a very ingenious way a theorem of Kahane-Salem-Zygmund on random trigonometric polynomials \cite[Theorem 4 in Chapter 6]{kahane1993some}, which gives (by the use of a probabilistic argument) the existence of homogeneous polynomials with  ``large coefficients'' and uniform norm  ``relatively small''.
This technique (and some refinements of it, for example \cite{defant2004maximum,bayart2012maximum}) do the work when dealing with upper bounds.

The lower bound is a horse of a different color. In \cite{defant2003bohr} Defant, Garc\'ia and Maestre related the Bohr radius with some non-elementary concepts of the local theory of Banach spaces: unconditionality in spaces of homogeneous polynomials via some Banach-Mazur distance estimates. Although at that moment this did not give optimal asymptotic bounds, it started a way with which $K(B_{\ell_p^n})$ would be obtained.

They were Defant, Frerick, Ortega-Cerd\'a, Ouna{\"i}es and Seip \cite{defant2011bohnenblust} who made an incredible contribution in the problem and managed to exhibit the exact asymptotic value of $K(B_{\ell_{\infty}^n})$.
The authors involved again into the game the classical Bohnenblust-Hille inequality, which was used to compute Bohr's convergence width eighty years before. This inequality asserts that the $\ell_{\frac{2m}{m+1}}$-norm of the coefficients of a given $m$-homogeneous polynomial in $n$-complex variables is bounded by a constant \emph{independent of $n$} times its supreme norm on the polydisk. Precisely, given $m \in \NN$. there is a constant $C_m>0$ such that for every $m$-homogeneous polynomial $P(z)=\sum_{\vert \alpha \vert = m} a_{\alpha} z^{\alpha}$ in $\CC^n$,

\begin{equation}
\left( \sum_{\vert \alpha \vert =m} \vert a_{\alpha} \vert^{\frac{2m}{m+1}} \right)^{\frac{m+1}{2m}}\leq C_m \sup_{z \in \DD^n} \vert P(z) \vert.
\end{equation}

The groundbreaking progress consisted in showing that $C_m$ is in fact hypercontractive; that is, $C_m$ can be taken less than or equal to $C^m$ for some absolute constant $C>0$. With this at hand they proved that $K(B_{\ell_{\infty}^n})$ behaves asymptotically as $\sqrt{\frac{\log(n)}{n}}$ (other cornerstone of the paper is that they have also described the  Sidon constant for the set of frequencies $\{\log (n) : n \text{ a positive integer } \leq N \}$).
This paper, arguably in some sense, marked the path of the whole area over the last years.
In fact much more can be said about $K(B_{\ell_{\infty}^n})$: Bayart, Pellegrino and Seoane \cite{bayart2014bohr} managed to push these techniques further in an amazingly ingenious way to obtain that  $\lim_{n\to \infty}  \frac{K(B_{\ell_{\infty}^n})}{\sqrt{\frac{\log(n)}{n}}}=1$.

Since $K(B_{\ell_{\infty}^n})$ bounds from below the radius $K(\mathcal R)$ for any other Reinhardt domain $\mathcal R$, the range where $p \geq 2$ easily follows. The solution of the case $p<2$, required  quite different  methods. A celebrated theorem proved independently by Pisier \cite{pisier1986factorization} and Sch{\"u}t \cite{schutt1978unconditionality} allows to study unconditional bases in spaces of multilinear forms in terms of some invariants such as the local unconditional structure or the Gordon-Lewis property.
These results have their counterpart in the context of spaces of polynomials as shown in \cite{defant2001unconditional}, replacing the full tensor product by the symmetric one.

Defant and Frerick \cite{defant2011bohr} (continuing their previous work given in \cite{defant2006logarithmic}) established some sort of extension of Pisier-Sch{\"u}t result to the symmetric tensor product with accurate bounds and gave a new estimate on the Gordon-Lewis constant of the symmetric tensor product. As a byproduct, they found the exact asymptotic growth for the Bohr radius on the unit ball of the spaces $\ell_p^n$.

The aforementioned results give the following relation for the Bohr radius.

\begin{theorem}{\cite{defant2011bohnenblust,defant2011bohr}} \label{crecimiento radio de Bohr} For $1 \leq p \leq \infty$, we have
\begin{equation}
K(B_{\ell_p^n}) \sim \left( \frac{\log(n)}{n}\right)^{1- \frac{1}{\min\{p,2\}}}.
\end{equation}

\end{theorem}

The proof of the exact asymptotic behavior of $K(B_{\ell_p^n})$ given in \cite{defant2011bohr} for $p<2$ as mentioned before use ``sophisticated machinery'' from the Banach space theory. Inspired by recent results from the
general theory of Dirichlet series, in \cite{bayart2016monomial} Bayart, Defant and Schl{\"u}ters managed to give upper estimates for the unconditional basis constants of spaces of polynomials on $\ell_p$ spanned by finite sets of monomials, which avoid the use of this ``machinery''. This perspective gives a new and, in a sense, clear proof of Theorem~\ref{crecimiento radio de Bohr} for the case $p<2$.

The study of the Bohr radius together with the techniques developed for this purpose have been enriching many mathematical areas such us number theory \cite{mathematicum2015dirichlet,defant2011bohnenblust}, complex analysis \cite{bayart2017multipliers,bayart2016monomial,defant2009domains}, operator algebras \cite{paulsen2002bohr,dixon1995banach},	 random polynomials \cite{boas2000majorant} (together with several works influenced by this article such as  \cite{defant2004maximum,bayart2012maximum,galicer2015vonNeumann}), the study of functions on the boolean cube \cite{defant2017bohr,defant2017fourier} (which are fundamental in theoretical computer science, graph theory, social choice, etc.) and even in quantum information \cite{defant2017fourier,montanaro2012some}. 

Our aim is to continue the study of the Bohr phenomenon for mixed Reinhardt domains.
Let $\mathcal R$ and $\mathcal S$ be two Reinhardt domain in $\CC^n$. The mixed Bohr radius $K(\mathcal R,\mathcal S)$ is defined as the biggest number $r\geq 0$ such that for every analytic function $f(z)=\sum_{\alpha} a_{\alpha} z^{\alpha}$ bounded on $\mathcal R$, it holds:

\begin{equation}
\sup_{z \in r \cdot \mathcal S} \sum_{\alpha} \vert a_{\alpha} z^{\alpha} \vert \leq \sup_{z \in \mathcal R} \vert f(z) \vert.
\end{equation}

We will focus in the case where $\mathcal R$ and $\mathcal S$ are the closed unit balls of $\ell_p$ and $\ell_q$ for $1 \leq p, q \leq \infty$. Note that $K(B_{\ell_p^n})$ in the previous notation is just $K(B_{\ell_p^n}, B_{\ell_p^n})$.
Our contribution is the following theorem which provides the correct asymptotic estimates for the full range of $p$'s and $q$'s.

\begin{theorem}\label{maintheorem}
	Let $1 \leq p, q \leq \infty$, with $q \neq 1$. The asymptotic growth of the $(p,q)$-Bohr radius is given by
		\[K(B_{\ell_p^n},B_{\ell_q^n})  \sim \begin{cases}
		1 & \text{ if (I): } 2 \leq p \leq \infty \; \wedge \; \frac{1}{2} + \frac{1}{p} \le \frac{1}{q},\\
		\frac{\sqrt{\log(n)}} {n^{\frac{1}{2} + \frac{1}{p} - \frac{1}{q}}} & \text{ if (II): } 2 \leq p \leq \infty \; \wedge \; \frac{1}{2} + \frac{1}{p} \ge \frac{1}{q}, \\
		\frac{\log(n)^{1 - \frac{1}{p}}}{n^{1-\frac{1}{q}}}  & \text{ if (III): } 1 \leq p , q \leq 2. \\
		\end{cases}
		\]

		For $q=1$ and every $1 \le p \le \infty$, $K(B_{\ell_p^n},B_{\ell_q^n})  \sim 1$.
\end{theorem}

	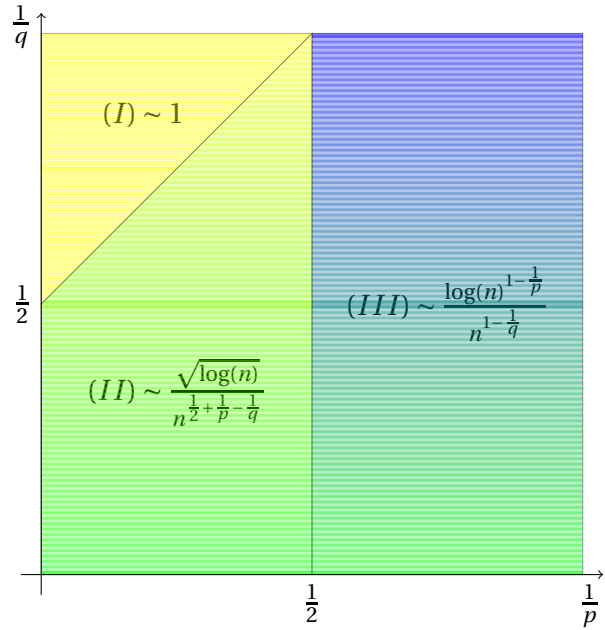
\begin{figure} \label{Figura radio}
		\begin{center}
			\begin{tikzpicture}[scale=0.9]
			%la escala esta hecha para m=4
			
			% (I)
			\draw (1.5,6.8) node {$(I) \sim 1$};
			\path[draw, shade, top color=yellow, bottom color=yellow, opacity=.3]
			(4,8) node[below] {$ $}  -- (0, 8) -- (0,4) -- cycle;
			
			% (II)
			\draw (2,2.7) node {$(II)\sim \frac{\sqrt{\log(n)}} {n^{\frac{1}{2} + \frac{1}{p} - \frac{1}{q}}}$};
			\path[draw, shade, top color=yellow, bottom color=green, opacity=.3]
			(0,0) node[below] {$ $}  -- (0, 4) -- (4,8) -- (4,0) -- cycle;

			% (III)
			\draw (6,4) node {$(III)\sim \frac{\log(n)^{1 - \frac{1}{p}}}{n^{1-\frac{1}{q}}}$};
			\path[draw, shade, top color=blue, bottom color=green, opacity=.3]
			(4,8) node[below] {$ $} -- (8,8) -- (8,0) node[below] {$ $}
			-- (4, 0) -- cycle;
			
%			% (III')
%			\draw[dashed] (4,4) -- (8,8);
%			\draw (5.4,6.7) node {$(III')$};
%			

		%	\draw[dotted] (0,4) -- (4,4); %1/q=1/2
			%\draw[dotted] (1,0) -- (1,4); %1/q=1/2m
			%\draw[dotted] (0,7.3) -- (4,7.3); %1/q=1/2
			
			\draw (4,0) node[below] {$\frac12$};
			\draw (0,4) node[left] {$\frac12$};
			%\draw (0,7.3) node[left] {$\frac{2m-1}{2m}$};
			%\draw (0,3.3) node[left] {$\frac{m-1}{2m}$};
			%\draw (2, 0) node[below] {$\frac1{m} $};
			%\draw (1, 0) node[below] {$\frac1{2m} $};

			% x-axis
			\draw (8.1, 0) node[below] {$\frac1{p} $};
			\draw[->] (-0.3,0) -- (8.3, 0);
			% y-axis
			\draw (0,8.1) node[left] {$\frac1{q} $};
			\draw[->] (0,-0.3) -- (0, 8.3);
			\end{tikzpicture}
		\end{center}

		\caption{Graphical overview of the mixed Bohr radius described in Theorem~\ref{maintheorem}.}
		
	\end{figure}
	
	As for $K(B_{\ell_p^n})$, the upper bound are obtained using random polynomials with adequate coefficients and relatively small norm \cite{boas2000majorant,defant2004maximum,bayart2012maximum}.
	To obtain the lower bounds the proof is divided in several cases. %(depending wether $ p $ is larger or smaller than $2$).
	For $p<2$ we have combined an appropriate  way to divide and distinguish certain subsets of monomials together with the upper estimates for the unconditional basis constants of spaces of polynomials on $\ell_p$ spanned by finite sets of monomials given in \cite{bayart2016monomial}.
	The interplay between  monomial convergence and mixed unconditionality for spaces of homogeneous polynomials presented in \cite[Theorem 5.1.]{defant2009domains} (which, of course, gives information on the Bohr radius) is crucial for the case $p>2$. We have strongly used some recent inclusion for the set of monomial convergence $\dom H_{\infty}(B_{\ell_p})$ $p \geq 2$ given in \cite{defant2009domains,bayart2017multipliers}.
	Therefore, it is worth noting that the techniques and results developed in the last years were fundamental for our proof.
	
	The article is organized as follows. In Section~\ref{Preliminaries} we present some basic background and results that we will use to prove Theorem~\ref{maintheorem}. We also give in this section some of the notation and concepts that appeared in this introduction. 
	Moreover, we include a heuristic argument as to why one should find, when studying the asymptotic behavior of the mixed Bohr radius, three differentiated regions in Theorem \ref{maintheorem} (see Figure \ref{Figura radio}). In Sections \ref{upper} and \ref{lower} we show the upper and lower estimates for the theorem respectively.
	In Sections \ref{upper} and \ref{lower} we show the upper and lower estimates for the theorem respectively.
	
	\section{Preliminaries}\label{Preliminaries}
	We write by $\DD$ the closed unit disk in the complex plane $\CC$.
	As usual we denote $\ell_{p}^{n}$ for the Banach space of all $n$-tuples $z=(z_1, \dots, z_n) \in \mathbb{C}^{n}$ endowed with the norm $\Vert (z_{1} , \ldots , z_{n}) \Vert_{p} = \Big( \sum_{i=1}^{n} \vert z_{i} \vert^{p} \Big)^{1/p}$ if $1 \leq p < \infty$, and $\Vert (z_{1} , \ldots , z_{n}) \Vert_{\infty} = \max_{i=1 , \ldots, n} \vert z_{i} \vert$ for $p = \infty$.  The unit ball of $\ell_{p}^{n}$ is denoted by $B_{\ell_p^n}$. For $1 \leq p \leq \infty$ we write $p'$ for its conjugate exponent (i.e., $\frac{1}{p}+\frac{1}{p'}=1$).

For every $x,y \in \CC^\NN$ we denote $|x|=(|x_1|, \ldots,|x_n|, \ldots )$, and $|x| \le |y|$ will mean that $|x_i| \le |y_i| $ for every $i \in \NN$. Recall that a Banach sequence space is a Banach space $(X, \| \cdot \|_X)$  with  $ \ell_1 \subset X \subset \ell_\infty$; and such that whenever $y \in X$, $x \in \CC^{\NN}$ and $|x| \le |y|$ it follows $x \in X$ and $\| x \|_X \le \| y \|_X$. A non-empty open set $\Rr \subset X$  is called a Reinhardt domain whenever given $ x \in \CC^\NN$ and $y \in \Rr$ such that $ |x| \le |y| $ then it holds $x \in \Rr $.

Given a Banach sequence space $X$ and fixed $n \in \NN$ its $n$-th projection $ X_n $ is defined as the quotient space induced by the mapping
\begin{align*}
\pi_n : X &\rightarrow \CC^n \\
		x & \mapsto (x_1, \ldots, x_n).
\end{align*}

	An $m$-homogeneous polynomial in $n$ variables is a function $P : \mathbb{C}^{n} \to \mathbb{C}$ of the form
	\[
	P(z_{1} , \ldots , z_{n}) = \sum_{\alpha \in \Lambda(m,n)}
	a_{\alpha} z^{\alpha},
	\]
	where $\Lambda(m,n):= \{ \alpha \in \mathbb{N}_0^n : |\alpha|:= \alpha_{1} + \cdots + \alpha_{n} = m \}$, $z^{\alpha}: = z_{1}^{\alpha_{1}} \cdots z_{n}^{\alpha_{n}}$ and $a_{\alpha} \in \mathbb{C}$. We will use the notation $a_\alpha =: a_\alpha(P)$.

	Another way of writing a polynomial $P$ is as follows:
	\[
	P(z_{1} , \ldots , z_{n}) = \sum_{\substack{\mathbf j \in \mathcal J(m,n)}}
	c_{\mathbf j} z_{\mathbf j},
	\]
	where  $\mathcal J(m,n):=\{ \mathbf j=(j_{1},\ldots , j_{k}) : 1 \leq j_{1} \leq \ldots \leq j_{k} \leq n \}$, $z_{\mathbf j}:=z_{j_{1}} \cdots z_{j_{k}}$ and $c_{\mathbf j} \in \mathbb{C}$. Note that $c_{\mathbf j}= a_{\alpha}$ with $\mathbf j=(1, \stackrel{\alpha_{1}}{\ldots}, 1, \ldots , n, \stackrel{\alpha_{n}}{\ldots} ,n)$. For some fixed $\jj \in \Jj(m,n)$ and some $ \ii = (i_1, \ldots, i_m) \in \NN^m$ we say $\ii \in [\jj] $ if the exists some permutation $\sigma \in S_m$ such that $(i_{\sigma(1)}, \ldots, i_{\sigma(m)}) = \jj$ and $|\jj|$ will denote the number of elements in $[\jj]$. Observe that $|\jj|=\frac{m!}{\alpha!}$ if $\mathbf j=(1, \stackrel{\alpha_{1}}{\ldots}, 1, \ldots , n, \stackrel{\alpha_{n}}{\ldots} ,n)$.
	
	The elements $(z^{\alpha})_{\alpha \in \Lambda(m,n)}$ (equivalently, $(z_{\mathbf j})_{\mathbf j \in \mathcal J(m,n)}$) are commonly refereed as \emph{the monomials}.
	
	Given a subset $\mathcal J \subset \mathcal J(m,n)$, we call
		$$\mathcal J^* = \{ \mathbf j \in \mathcal J(m-1,n) : \text{ there is } k\geq 1, (\mathbf j,k) \in \mathcal J\}.$$
	
	For $1 \leq p
	\leq \infty$ we denote by $\mathcal{P} (^{m}\ell_{p}^{n} )$ the Banach space of all $m$-homogeneous
	polynomials in $n$ complex variables equipped with the uniform (or sup) norm
	\[
	\Vert P \Vert_{\mathcal{P} (^{m}\ell_{p}^{n} )} := \sup_{z \in B_{\ell_p^n}} \big|P(z) \big|.
	\]

	Given two Banach sequence spaces $X$ and $Y$, for $ n,m \in \mathbb{N}$ let $ \chi_M(\Pp(^m {X_n}),\Pp(^m {Y_n}))$ be the best constant $\lambda > 0$ such that $$ \sup_{z \in {B_{Y_n}}}\left| \displaystyle\sum_{ \alpha \in \Lambda(m,n)} \theta_\alpha a_\alpha z^\alpha \right| \leq \lambda \sup_{z \in {B_{X_n}}}\left| \displaystyle\sum_{ \alpha \in \Lambda(m,n)} a_\alpha z^\alpha \right|,  $$ for every $ ( a_\alpha )_{\alpha \in \Lambda(m,n) } \subset \CC $ and every choice of complex numbers $(\theta_\alpha)_{\alpha \in \Lambda(m,n)}$ of modulus one.

	When $X=\ell_p$ and $Y=\ell_q$   we will denote $\chi_M(\Pp(^m {X_n}),\Pp(^m {Y_n}))$, the $(p,q)$-mixed unconditionally constant for the monomial basis of $\PP(^m \CC^n)$, as $\chi_M(\Pp(^m {\ell_p^n}),\Pp(^m {\ell_q^n}))$. It should be mentioned that, for any fixed $m \in \mathbb{N}$, the asymptotic growth of $\chi_M(\Pp(^m {\ell_p^n}),\Pp(^m {\ell_q^n}))$ as $n\to\infty$ was studied in \cite{galicer2016sup}.
	
	Every entire function $f: \CC^n \rightarrow \CC$ can be written as
	\[
	f = \dis\sum_{m \ge 0} \dis\sum_{\alpha \in \Lambda(m,n)} a_\alpha(f) z^\alpha.
	\]
	Recall that $K(B_{\ell_p^n},B_{\ell_q^n}) $ stands for the $n$-dimensional $(p,q)$-Bohr radius. That is, $K(B_{\ell_p^n},B_{\ell_q^n}) $ denotes the greatest constant $r>0$ such that for every entire function $f=\sum_{\alpha} a_{\alpha} z^{\alpha}$ in $n$-complex variables, we have the following (mixed) Bohr-type  inequality
	$$\sup_{z \in  r \cdot B_{\ell_q^n}} \sum_{\alpha} \vert a_{\alpha} z^{\alpha} \vert \leq \sup_{z \in  B_{\ell_p^n}} \vert f(z) \vert.$$
	
	In the same way, the $m$-homogeneous mixed Bohr radius, $K_m(B_{\ell_p^n},B_{\ell_q^n})$, is defined as the greatest $r>0$ such that for every $P \in \PP(^m \CC^n)$ it follows
	\[
	\dis\sup_{z \in B_{\ell_q^n}} \dis\sum_{\alpha \in \Lambda(m,n)} |a_{\alpha} z^\alpha| r^m = \dis\sup_{z \in r \cdot B_{\ell_q^n}} \dis\sum_{\alpha \in \Lambda(m,n)} |a_{\alpha} z^\alpha| \le \| P \|_{\PP(^m \ell_p^n)}
	\]
	
	It is plain that $K(B_{\ell_p^n},B_{\ell_q^n})  \le K_m(B_{\ell_p^n},B_{\ell_q^n}) $.
	
	\begin{remark}\label{hom Boh vs unc}
		 $$ K_m(B_{\ell_p^n},B_{\ell_q^n}) = \frac{1}{\chi_M(\Pp(^m {\ell_p^n}),\Pp(^m {\ell_q^n}))^{1/m}}.$$
	\end{remark}
	
	 \begin{proof}
	 Given $P \in \PP (^m \CC^n)$ and for any $( \theta_\alpha )_{\alpha \in \Lambda(m,n)}$ we have
		\begin{align*}
		\| \dis\sum_{\alpha \in \Lambda(m,n)} \theta_\alpha a_\alpha(P) z^\alpha \|_{\PP(^m \ell_q^n)}
		& \le \| \dis\sum_{\alpha \in \Lambda(m,n)} |a_\alpha(P) z^\alpha| \|_{\PP(^m \ell_q^n)} \\
		& = \| \dis\sum_{\alpha \in \Lambda(m,n)} |a_\alpha(P) z^\alpha| (K_m(B_{\ell_p^n},B_{\ell_q^n}))^m \|_{\PP(^m \ell_q^n)}  \frac{1}{(K_m(B_{\ell_p^n},B_{\ell_q^n}))^m} \\
		& \le \frac{1}{(K_m(B_{\ell_p^n},B_{\ell_q^n}))^m} \| P \|_{\PP(^m \ell_p^n)},
		\end{align*}
		which leads to the inequality $ \chi_M(\Pp(^m {\ell_p^n}),\Pp(^m {\ell_q^n}))^{1/m} \le \frac{1}{K_m(B_{\ell_p^n},B_{\ell_q^n})}$. On the other hand, for $P \in \PP (^m \CC^n)$ take $\theta_\alpha = \frac{\overline{a_\alpha(P)}}{|a_\alpha(P)|}$. Then we have
		\begin{align*}
		\| \dis\sum_{\alpha \in \Lambda(m,n)} |a_\alpha(P) z^\alpha| \|_{\PP(^m \ell_q^n)}
		& = \| \dis\sum_{\alpha \in \Lambda(m,n)} \theta_\alpha a_\alpha(P) z^\alpha \|_{\PP(^m \ell_q^n)}  \\
		& \le \chi_M(\Pp(^m {\ell_p^n}),\Pp(^m {\ell_q^n})) \| P \|_{\PP(^m \ell_p^n)},
		\end{align*}
		or equivalently,
		\[
		\dis\sup_{z \in B_{\ell_q^n}} \dis\sum_{\alpha \in \Lambda(m,n)} |a_{\alpha} z^\alpha| \left( \frac{1}{\chi_M(\Pp(^m {\ell_p^n}),\Pp(^m {\ell_q^n}))^{1/m}}\right)^m \le \| P \|_{\PP(^m \CC^n)},
		\]
		which means $ \frac{1}{ \chi_M(\Pp(^m {\ell_p^n}),\Pp(^m {\ell_q^n}))^{1/m} } \le K_m(B_{\ell_p^n},B_{\ell_q^n}) $.
	\end{proof}
	It will be useful to remember a classic result due to F. Wiener (see \cite{khavinson1997bohr}) which asserts that for every holomorphic function $f$ written as the sum of $m$-homogeneous polynomials as $f = \dis\sum_{m \ge 1} P_m + a_0$ and such that $\sup_{z \in B_{\ell_p^n }} |f(z)| \le 1$ it holds
		\begin{equation}\label{Winer}
	\| P_m \|_{\PP(^m \ell_p^n)} \le 1 - |a_0|^2,
	\end{equation}
		for every $m \in \zN$.
	
	In general this inequality is presented for the uniform norm on the polydisk $\Vert \cdot \Vert_{\PP(^m \ell_{\infty}^n)}$ (i.e., $p=\infty$), but this version easily follows by a standard modification of the original argument (given $z \in B_{\ell_p^n}$ consider the auxiliary function $g : \CC^n \to \CC$ given by $g(w):=f(w\cdot z)$).
	
	The next lemma is an adaption of the case $p=q$, see \cite[Theorem 2.2.]{defant2003bohr} and constitutes the basic link between Bohr radius and unconditional
basis constants of spaces of polynomials on the mixed context ($p$ not necessarily equal to $q$).
	\begin{lemma}\label{Bohr vs unc}
		For every $n \in \zN$ and $1 \le p , q \le \infty$ it holds
		\[
		\frac{1}{3} \frac{1}{ \sup_{ m \ge 1}  \chi_M(\Pp(^m {\ell_p^n}),\Pp(^m {\ell_q^n}))^{1/m}}  \le K(B_{\ell_p^n},B_{\ell_q^n})  \le \min \left\lbrace \frac{1}{3}, \frac{1}{\sup_{ m \ge 1}  \chi_M(\Pp(^m {\ell_p^n}),\Pp(^m {\ell_q^n}))^{1/m}} \right\rbrace.
		\]
	\end{lemma}
	
	\begin{proof}
		Form Remark \ref{hom Boh vs unc} we have $K_{p,q(n)} \le \inf_{ m \ge 1}\frac{1}{ \chi_M(\Pp(^m {\ell_p^n}),\Pp(^m {\ell_q^n}))^{1/m}}=  \frac{1}{\sup_{ m \ge 1}  \chi_M(\Pp(^m {\ell_p^n}),\Pp(^m {\ell_q^n}))^{1/m}}$ and due to Bohr's inequality we know $K(\DD) = \frac{1}{3}$ as it is clear that $ K(B_{\ell_p^n},B_{\ell_q^n})  \le K(\DD)$ for every $n \in \zN$ the right hand side inequality holds.
		For the left hand side inequality let us take some holomorphic function $f$, without loss of generality let us assume $ \sup_{z \in B_{\ell_p^n}}|f(z)| \le 1$, and consider its decomposition as a sum of $m$-homogeneous polynomials $f = \dis\sum_{m \ge 0} P_m$. For every $m \in \zN_0$ it holds $ P_m(z) = \dis\sum_{\alpha \in \Lambda(m,n)} a_\alpha(f) z^\alpha$, thus taking $\rho = \sup_{m \ge 1 } \chi_M(\Pp(^m {\ell_p^n}),\Pp(^m {\ell_q^n}))^{1/m} $ and using Remark \ref{hom Boh vs unc} again it follows
		\begin{equation*}\label{desig hom}
		\Big\| \dis\sum_{\alpha \in \Lambda(m,n)} |a_\alpha(f)| \left(\frac{z}{\rho}\right)^\alpha \Big\|_{\PP(^m \ell_q^n)} \le  \Big\| \dis\sum_{\alpha \in \Lambda(m,n)} a_\alpha(f) z^\alpha \Big\|_{\PP(^m \ell_p^n)}.
		\end{equation*}
		
		Applying the above mentioned Wiener's result for some $w \in B_{\ell_q^n}$ we have that
		\begin{align*}
		\dis\sum_{m \ge 0} \dis\sum_{\alpha \in \Lambda(m,n)} |a_\alpha(f)| \left( \frac{w}{3 \rho} \right)^\alpha
		& \le |a_0(f)| + \dis\sum_{m \ge 1} \frac{1}{3^m} \Big\| \dis\sum_{\alpha \in \Lambda(m,n)} a_\alpha(f) z^\alpha \Big\|_{\PP(^m \ell_p^n)} \\
		& \le |a_0(f)| + \dis\sum_{m \ge 1} \frac{1}{3^m}(1 - |a_0(f)|^2) \\
		& \le |a_0(f)| + \frac{1 - |a_0(f)|^2}{2} \le 1 ,
		\end{align*}
		where last inequality holds as $ |a_0(f)| \le \sup_{z \in B_{\ell_p^n}} |f(z)| \le 1$. The last chain of inequalities and the maximality of mixed Bohr radius lead us to $\frac{1}{3 \rho} \le K(B_{\ell_p^n},B_{\ell_q^n}) $ as we wanted to prove.
	\end{proof}

The previous lemma shows that understanding $K(B_{\ell_p^n},B_{\ell_q^n})$ translates into seeing how the constant $\chi_M(\Pp(^m {\ell_p^n}),\Pp(^m {\ell_q^n}))^{1/m}$ behaves. 
It should be mentioned that, for any fixed $m \in \mathbb{N}$, the asymptotic growth of $\chi_M(\Pp(^m {\ell_p^n}),\Pp(^m {\ell_q^n}))$ as $n\to\infty$ was studied in \cite{galicer2016sup}.	
These results unfortunately are not useful because, as can be seen in Lemma \ref{Bohr vs unc}, one needs to comprehend how $\chi_M(\Pp(^m {\ell_p^n}),\Pp(^m {\ell_q^n}))^{1/m}$ grows by moving both the number of variables, $n$, and the degree of homogeneity, $m$. But beyond this, they give a guideline of what to expect (at least what the different regions in Figure \ref{Figura radio} should look like).
Indeed, in \cite{galicer2016sup} we have proved the following

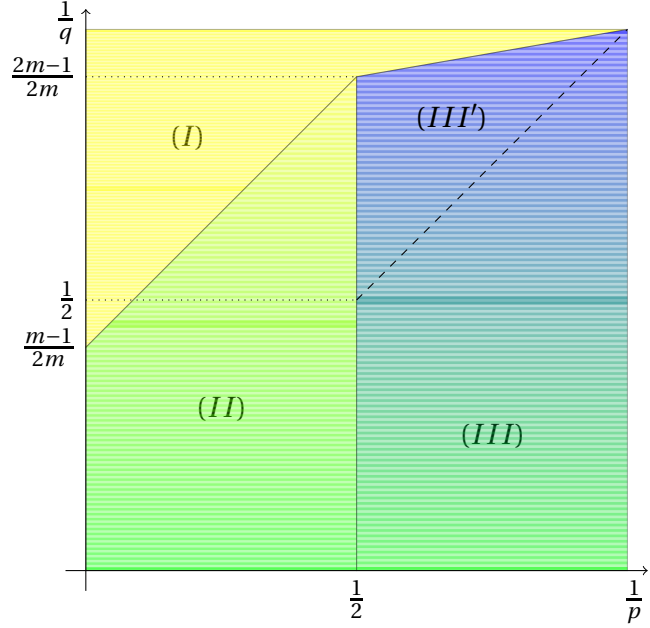
\begin{figure}  \label{Figura incondicionalidad}
\begin{center}
\begin{tikzpicture}[scale=0.9]
%la escala esta hecha para m=4

% (I)
	\draw (1.5,6.4) node {$(I)$};
  \path[draw, shade, top color=yellow, bottom color=yellow, opacity=.3]
     (4,7.3) node[below] {$ $}  -- (8, 8) -- (0, 8) -- (0,3.3) -- cycle;

% (II)
    \draw (2,2.4) node {$(II)$};
  \path[draw, shade, top color=yellow, bottom color=green, opacity=.3]
     (0,0) node[below] {$ $}  -- (0, 3.3) -- (4,7.3) -- (4,0) -- cycle;

% (III)
	\draw (6,2) node {$(III)$};
  \path[draw, shade, top color=blue, bottom color=green, opacity=.3]
    (4,7.3) node[below] {$ $} -- (8,8) -- (8,0) node[below] {$ $}
     -- (4, 0) -- cycle;

% (III')
\draw[dashed] (4,4) -- (8,8);
	\draw (5.4,6.7) node {$(III')$};

\draw[dotted] (0,4) -- (4,4); %1/q=1/m
%\draw[dotted] (1,0) -- (1,4); %1/q=1/2m
\draw[dotted] (0,7.3) -- (4,7.3); %1/q=1/2

\draw (4,0) node[below] {$\frac12$};
\draw (0,4) node[left] {$\frac12$};
\draw (0,7.3) node[left] {$\frac{2m-1}{2m}$};
\draw (0,3.3) node[left] {$\frac{m-1}{2m}$};
%\draw (2, 0) node[below] {$\frac1{m} $};
%\draw (1, 0) node[below] {$\frac1{2m} $};

  % x-axis
\draw (8.1, 0) node[below] {$\frac1{p} $};
  \draw[->] (-0.3,0) -- (8.3, 0);
  % y-axis
\draw (0,8.1) node[left] {$\frac1{q} $};
  \draw[->] (0,-0.3) -- (0, 8.3);
 \end{tikzpicture}
\end{center}

\caption{Graphical overview of the mixed unconditional constant described in Theorem~\ref{cte incond mon}.}

\end{figure}
\begin{theorem}\label{cte incond mon}
\[
\begin{cases}  \; \chi_{p,q}(\mathcal{P}(^m \mathbb{C}^n)) \sim 1 & \text{ for } (I): \; [\frac{1}{p} + \frac{m-1}{2m} \leq \frac{1}{q} \wedge \frac{1}{p} \leq \frac{1}{2} ] \text{ or } [ \frac{m-1}{m} + \frac{1}{mp} < \frac{1}{q} \wedge \frac{1}{2} \leq \frac{1}{p} ], \\
\;\chi_{p,q}(\mathcal{P}(^m \mathbb{C}^n)) \sim n^{m (\frac{1}{2}+\frac{1}{p}-\frac{1}{q}) - \frac{1}{2}} & \text{ for } (II) \; \; [ \frac{1}{p} + \frac{m-1}{2m} \geq \frac{1}{q} \wedge \frac{1}{p} \leq \frac{1}{2}  ] ,\\
\; \chi_{p,q}(\mathcal{P}(^m \mathbb{C}^n)) \sim n^{(m-1)(1-\frac{1}{q}) + \frac{1}{p}  -\frac{1}{q}} & \text{ for } (III) \; : \; [ \frac{1}{p} \geq \frac{1}{q} \; \wedge \; \frac{1}{2} \leq \frac{1}{p}  ], \\
\; \chi_{p,q}(\mathcal{P}(^m \mathbb{C}^n)) \sim_{\varepsilon}  n^{(m-1)(1-\frac{1}{q}) + \frac{1}{p}  -\frac{1}{q}} & \text{ for } (III') \; : \; [ 1 - \frac{1}{m} + \frac{1}{mp} \geq \frac{1}{q}\ge\frac1{p} \; \wedge \; \frac{1}{2} < \frac{1}{p} <1 ]. \\
\end{cases}
\]
where $\chi_{p,q}(\mathcal{P}(^m \mathbb{C}^n)) \sim_{\varepsilon} n^{(m-1)(1-\frac{1}{q}) + \frac{1}{p}  -\frac{1}{q}}$ means that
$$  n^{(m-1)(1-\frac{1}{q}) + \frac{1}{p}  -\frac{1}{q}} \ll \chi_{p,q}(\mathcal{P}(^m \mathbb{C}^n)) \ll n^{(m-1)(1-\frac{1}{q}) + \frac{1}{p}  -\frac{1}{q} + \varepsilon},$$
for every $\varepsilon >0$.

%Moreover, in $(III')$ for every $\lambda > \frac{1}{p}$ we have
%$$  n^{(m-1)(1-\frac{1}{q}) + \frac{1}{p}  -\frac{1}{q}} \ll \chi_{p,q}(\mathcal{P}(^m \mathbb{C}^n)) \ll \log(n)^{m(\frac{1}{q} - \frac{1}{p}) + (\lambda + \frac{1}{p} ) \frac{m^2}{m-1} \frac{p-q}{q(p-1)}}   n^{(m-1)(1-\frac{1}{q}) + \frac{1}{p}  -\frac{1}{q}}.$$
%Figure~\ref{Figura incondicionalidad} represents the results stated above.
\end{theorem}

The heuristic to interpret the different regions in Theorem \ref{maintheorem} is the following:
If one thinks that the correct order of the region $(III')$ coincides with that of $(III)$ in Theorem \ref{cte incond mon} (in fact, this is what we believe) and assumes that the homogeneity degree is very very large ($ m \to \infty $) then the graph in Figure \ref{Figura incondicionalidad} transforms into the one presented in Figure \ref{Figura radio}. All this, together with the upper bounds that one gets after using classical random polynomials (see Section \ref{upper}, somehow the easy part) helped us to define where to aim to prove lower bounds.
We highlight that, the logarithmic factors that appear in Theorem \ref{maintheorem}, are missing in Theorem \ref{cte incond mon}. This is, somehow, not coincidental and their presence is due to the interplay between the number of variables and the degree of homogeneity in Lemma \ref{Bohr vs unc}.

We continue with some definitions that will be useful later.
Given $X$ a Banach sequence space we denote $H_{\infty}(B_X)$ to the space of holomorphic functions over $B_X$ endowed with the norm given by $\| f \|_{H_{\infty}(B_X)} : = \dis\sup_{z \in B_X}|f(z)|$. The domain of monomial convergence of $H_{\infty}(B_X)$ is defined as
\[
\dom (H_{\infty}(B_X)) := \Big\lbrace z \in \ell_\infty : \dis\sum_{m \ge 0} \dis\sum_{\alpha \in \Lambda(m,n)} |a_\alpha(f) z^\alpha| < \infty \mbox{ for every } f \in H_{\infty}(B_X) \Big\rbrace.
\]

The next theorem appears in \cite[Theorem $5$.$1$]{defant2009domains} and relates monomial convergence with the study of the mixed unconditional constant of the monomial basis.

\begin{theorem}\label{general conv mon}
For a couple of Banach sequence spaces $X,Y$ the following are equivalent
\begin{itemize}
\item[(1)] $r B_Y \subset \dom H_{\infty}(B_X)$ for some $r > 0$.
\item[(2)] There exists a constant $ C > 0 $ independent of $m$ such that
\[
\sup_{n \ge 1} \chi_M(\Pp(^m X_n), \Pp(^m Y_n)) \le C^m.
\]
\end{itemize}
\end{theorem}

	If $(a_{n})_{n}$ and $(b_{n})_{n}$ are two sequences of real numbers we will write $a_{n} \ll b_{n}$ if there
	exists a constant $C>0$ (independent of $n$) such that $a_{n} \leq C b_{n}$ for every $n$.
	We will write $a_{n} \sim b_{n}$ if $a_{n} \ll b_{n}$ and $b_{n} \ll a_{n}$.
	We will use repeatedly the Stirling formula which asserts
	%$\dis\lim_{n \to \infty} \frac{n!}{\sqrt{2 \pi n} \left( \frac{n}{e} \right)^n} = 1$ and implies
	\begin{equation}\label{Stirling}
	n! \sim \sqrt{2 \pi n} \left( \frac{n}{e} \right)^n.
	\end{equation}
	
	We use the letters $C, C_1, C_2$, etc. to denote absolute positive constants (which from one inequality to the other may vary and sometimes denoted in the same way).

\section{Upper bounds}\label{upper}
Upper bounds constitute the easy part: we will use the classical probabilistic approach. Bayart in \cite[Corollary 3.2]{bayart2012maximum} (see also \cite{boas2000majorant,defant2003bohr,defant2004maximum}) exhibited polynomials %with unimodular coefficients and
with small sup-norm on the unit ball of $\ell_p^n$.
He showed that for each $1 \leq p \leq \infty$ there exists an $m$-homogeneous polynomial in $n$ complex variables, $P(z):=  \sum_{\alpha \in \Lambda(m,n)} \varepsilon_{\alpha} \frac{m!}{\alpha!} z^{\alpha}$, with $\varepsilon_{\alpha} = \pm 1$ for every $\alpha$,   such that $\Vert P \Vert_{\mathcal{P} (^{m}\ell_{p}^{n} )} \leq D_p(m,n)$ where
\begin{equation} \label{polinomios de Bayart}
D_p(m,n):= C_p  \times
\begin{cases}
  (\log(m) m! )^{1- \frac{1}{p}} \;n^{1- \frac{1}{p}}& \text{ if } 1 \leq p \leq 2, \\
 (\log(m) m! )^{\frac{1}{2}} n^{m(\frac{1}{2} - \frac{1}{p}) + \frac{1}{2}} & \text{ if } 2 \leq p \leq \infty, \\
\end{cases}
\end{equation}
and  $C_p$  depends  exclusively on  $p$.

We will also need the following remark which is an easy calculus exercise.
\begin{remark}\label{opt f(m)}
For every positive numbers $a,b>0$ and $n \in \zN$, the function $f : \RR_{>0}  \rightarrow \RR$ given by $f(x) = x^a n^{\frac{b}{x}}$ attains its minimum at $x = \log(n)\frac{b}{a}$.
\end{remark}

\begin{proof}[Proof of the upper bounds of Theorem \ref{maintheorem}]
Upper bounds for the case $ \frac{1}{p} + \frac{1}{2} \le \frac{1}{q} $ in  Theorem \ref{maintheorem} are trivial.

Suppose $  \frac{1}{p} + \frac{1}{2} \ge \frac{1}{q}$ and let $(\varepsilon_\alpha)_{\alpha \in \Lambda(m,n)} \subset \lbrace -1, 1 \rbrace $ signs such that
\[
 \Big\| \dis\sum_{\alpha \in \Lambda(m,n)} \varepsilon_\alpha \frac{m!}{\alpha!} z^\alpha \Big\|_{B_{\ell_p^n}} \le D_p(m,n),
 \]
 as \eqref{polinomios de Bayart}. Taking $z_0=(\frac{1}{n^{1/q}}, \ldots, \frac{1}{n^{1/q}}) \in B_{\ell_q}$ we can conclude that
\begin{align*}
n^{m(1-\frac{1}{q})} & = \dis\sum_{ \alpha \in \Lambda(m,n)} \frac{m!}{\alpha!} \left( \frac{1}{n^\frac{1}{q}} \right)^m \\
& \le \Big\| \dis\sum_{ \alpha \in \Lambda(m,n)} | \varepsilon_\alpha | \frac{m!}{\alpha!} z^\alpha \Big\|_{B_{\ell_q^n}} \\
& \le \chi_M(\Pp(^m {\ell_p^n}),\Pp(^m {\ell_q^n})) \;  \Big\| \dis\sum_{\alpha \in \Lambda(m,n)} \varepsilon_\alpha \frac{m!}{\alpha!} z^\alpha \Big\|_{B_{\ell_p^n}} \\
& \le \chi_M(\Pp(^m {\ell_p^n}),\Pp(^m {\ell_q^n})) \cdot D_p(m,n).
\end{align*}
For  $1  \le p \le  2$ we have by Stirling formula \eqref{Stirling},
\begin{align*}
 \frac{1}{\chi_M(\Pp(^m {\ell_p^n}),\Pp(^m {\ell_q^n}))^{1/m}}
  & \le  \Big( C_p n^\frac{1}{p'} \; (\log(m) m!)^{\frac{1}{p'}} \Big)^{1/m} \frac{1}{n^{\frac{1}{q'}}} \\
  & \le C \frac{1}{n^{\frac{1}{q'}}} m^{\frac{1}{p'}} n^\frac{1}{p'm},
\end{align*}
where $C>0$ depends only on $p$.
Thanks to Lemma \ref{hom Boh vs unc}, Remark \ref{opt f(m)} and the previous inequality
\[
K(B_{\ell_p^n},B_{\ell_q^n})  \le C \frac{1}{n^{\frac{1}{q'}}} \dis\inf_{m \ge 1} m^{\frac{1}{p'}} n^{\frac{1}{2m}} \le C(p,q) \frac{\log(n)^{\frac{1}{p'}}}{n^{\frac{1}{q'}}}.
\]
%where we also use that $ n^\frac{a}{\log(n)}$ is bounded for any $a$.

On the other hand, for $ p \ge 2$  and $ \frac{1}{q}  \le \frac{1}{p} + \frac{1}{2}$ it follows
\begin{align*}
 \frac{1}{\chi_M(\Pp(^m {\ell_p^n}),\Pp(^m {\ell_q^n})))^{1/m}}
  & \le \Big( C_p  n^{\frac{1}{2}} \; (\log(m) m!)^{\frac{1}{2}} \Big)^{1/m} \frac{1}{n^{\frac{1}{2} + \frac{1}{p} - \frac{1}{q}}}  \\
  & \le C \frac{1}{n^{\frac{1}{2} + \frac{1}{p} - \frac{1}{q}}} m^{\frac{1}{2}} n^\frac{1}{2m}.
\end{align*}
Thus minimizing $m^{\frac{1}{2}} n^\frac{1}{2m}$ as in the previous case we get,
\[
K(B_{\ell_p^n},B_{\ell_q^n})  \le C \frac{\sqrt{\log(n)}} {n^{\frac{1}{2} + \frac{1}{p} - \frac{1}{q}}},
\]
as we wanted to prove.
\end{proof}

\section{Lower bounds} \label{lower}
For the proof of the lower bounds we need to consider four different cases. We begin with the case $q=1$ and the case $ p \le q $, which are the easy ones. Then we study the case   $1< q \le p \le 2$ where we use tools from unconditionality and finally the case $p\ge 2$ where the key tool is monomial convergence.

\subsection{The case  $  q=1 $}
By \cite{aizenberg2000multidimensional},
\begin{equation*}\label{p=q=1}
K(B_{\ell_1^n}) \sim 1.
\end{equation*}
Thus, for any $f(z) = \sum_{\alpha} a_{\alpha} z^\alpha $, it follows that
\begin{align*}
\dis\sup_{z\in K(B_{\ell_1^n}) \cdot B_{\ell_1^n}} \sum_{\alpha} |a_\alpha z^\alpha | & \le \sup_{z \in B_{\ell_1^n}} | f(z)|
\le  \sup_{z \in B_{\ell_p^n}} | f(z)| ,\\
\end{align*}
which implies that $K(B_{\ell_p^n},B_{\ell_1^n}) \ge K(B_{\ell_1^n})\sim 1$.

\subsection{The case  $ p \le q $}
For this case we will strongly use Theorem \ref{crecimiento radio de Bohr}.
The case $ p \le q $ is an easy corollary of this result.
\begin{proof}[Proof for the lower bound of Theorem~\ref{maintheorem}: the case $ p \le q $.]
Taking $m \in \NN$, for any $P(z) = \dis\sum_{ \alpha \in \Lambda(m,n) } a_{\alpha} z^\alpha $, it follows that
\begin{align*}
\Big\| \dis\sum_{\alpha \in \Lambda(m,n)} |a_\alpha| z^\alpha \Big\|_{\Pp(^m \ell_q^n)} & \le n^{m(\frac{1}{p} - \frac{1}{q})}  \Big\| \dis\sum_{\alpha \in \Lambda(m,n)} |a_\alpha| z^\alpha \Big\|_{\Pp(^m \ell_p^n)} \\
& \le n^{m(\frac{1}{p} - \frac{1}{q})} K(B_{\ell_p^n})^{-m}  \Big\| \dis\sum_{\alpha \in \Lambda(m,n)} a_\alpha z^\alpha \Big\|_{\Pp(^m \ell_p^n)} ,
\end{align*}
which implies that $ K_m(B_{\ell_p^n},B_{\ell_q^n}) \ge K(B_{\ell_p^n}) n^{\frac{1}{q} - \frac{1}{p} } $ for every $m \in \NN$. Using Lemma~\ref{Bohr vs unc} and Theorem~\ref{crecimiento radio de Bohr} we have, for $p\le 2$,
\[
K(B_{\ell_p^n},B_{\ell_q^n})  \ge \frac{1}{3} n^{\frac{1}{q}-\frac{1}{p}} K(B_{\ell_p^n}) \sim n^{\frac{1}{q}-\frac{1}{p}} \left( \frac{\log(n)}{n} \right)^{1 - \frac{1}{p}} = \frac{\log(n)^{1 - 1/p}}{n^{1 - 1/q}},
\]
and, for $p\ge 2$,
%\[
%K(B_{\ell_p^n},B_{\ell_q^n})  \ge \frac{1}{3} n^{\frac{1}{q}-\frac{1}{p}} K(B_{\ell_p^n},B_{\ell_p^n}) \sim
%\begin{cases}
% \frac{\log(n)^{\frac{1}{p'}}}{ n^{\frac{1}{q'}}} & \text{ if } 1 \leq p \leq 2, \\
% \frac{\sqrt{\log(n)}} {n^{\frac{1}{2} + \frac{1}{p} - \frac{1}{q}}} & \text{ if } 2 \leq p . \\
%\end{cases}.
%\]
$$
K(B_{\ell_p^n},B_{\ell_q^n})  \ge \frac{1}{3} n^{\frac{1}{q}-\frac{1}{p}} K(B_{\ell_p^n})  \sim n^{\frac{1}{q}-\frac{1}{p}} \left( \frac{\log(n)}{n} \right)^{1 - \frac{1}{2}} =\frac{\sqrt{\log(n)}} {n^{\frac{1}{2} + \frac{1}{p} - \frac{1}{q}}},
$$ which concludes the proof.
\end{proof}

%Note that the previous proof can also be used for the case $p \ge 2$ when $p \le q$. If $q \le p$ this tecnhique does not work any more and those cases require a more detailed treatment, that treatment cover the case $ p \le q $ when $2 \le p$ so we will give that proof.

\subsection{The case $1 < q \le p \leq 2$}

\begin{lemma}{\cite[Lemma 3.5.]{bayart2016monomial}}\label{LEMPOLY}
 Let $1 \leq p \leq \infty$ and  $P$ be an $m$-homogeneous polynomial in $\mathcal P(^m\ell_p^n)$. Then for any $\bj\in \jbn{m-1}$
$$\left(\sum_{k=j_{m-1}}^n |c_{(\bj,k)}(P)|^{p'}\right)^{1/p'}\leq m e^{1+\frac{m-1}r}{|\bj|}^{1/p} \|P\|_{\mathcal P(^m\ell_p^n)}.$$
\end{lemma}
The next lemma is an adaptation of the case $r\le 2$ in \cite[Theorem 3.2]{bayart2016monomial} to the mixed context for our purposes.
\begin{lemma}\label{cteincondicionalidad}
Let $1 \leq q \leq p \leq 2$. Then we have
$$\chi_M(\Pp(^m {\ell_p^n}),\Pp(^m {\ell_q^n})) \le m e^{1+\frac{m-1}p} \left(\sum_{\bj\in \mathcal{J}(m-1,n)} \vert \bj \vert^{(1/p - 1/q)q'}\right)^{1/q'}.$$
\end{lemma}
\begin{proof}
    Fix
    $P\in\mathcal P(^m\ell_p^n) $ and  $u \in \ell_q^n$. Then, by Lemma \ref{LEMPOLY}, for any $\bj\in \mathcal{J}(m,n)^*$,
    $$\left(\sum_{k:\ (\bj,k)\in \mathcal{J}(m,n)} |c_{(\bj,k)}(P)|^{p'}\right)^{1/p'}\leq\left(\sum_{k=j_{m-1}}^n |c_{(\bj,k)}(P)|^{p'}\right)^{1/p'}\leq m e^{1+\frac{m-1}q}{|\bj|}^{1/p} \|P\|_{\mathcal P(^m\ell_p^n)}.$$
     Now applying the above inequality, H\"older's inequality (two times) %, the fact that $1 \leq q \leq p \leq 2$
and the multinomial formula we have
    \begin{eqnarray*}
         \sum_{\bj\in \mathcal{J}(m,n)} |c_{\bj}(P)| |u_\bj| &=&
         \sum_{\bj\in \mathcal{J}(m,n)^*} \left(\sum_{k:\ (\bj,k)\in \mathcal{J}(m,n)} |c_{(\bj,k)}| |u_\bj| |u_k| \right)\\&\leq&
        \sum_{\bj\in \mathcal{J}(m,n)^*}|u_\bj| \,\, \left(\sum_{k:\ (\bj,k)\in \mathcal{J}(m,n)} |c_{(\bj,k)}|^{q'}\right)^{1/q'}
        \left(\sum_{k} |u_k|^q\right)^{1/q}\\
        &\leq& \sum_{\bj\in \mathcal{J}(m,n)^*}|u_\bj| \,\, \left(\sum_{k:\ (\bj,k)\in \mathcal{J}(m,n)} |c_{(\bj,k)}|^{p'}\right)^{1/p'}
       \|u\|_q\\
         &\leq&m e^{1+\frac{m-1}p}\sum_{\bj\in \mathcal{J}(m,n)^*} {|\bj|}^{1/p}|u_\bj| \|u\|_q \|P\|_{\mathcal P(^m\ell_p^n)}\\
        &\leq&m e^{1+\frac{m-1}p} \left(\sum_{\bj\in \mathcal{J}(m,n)^*} {|\bj|}|u_\bj|^q\right)^{1/q}\left(\sum_{\bj\in \mathcal{J}(m,n)^*} \vert \bj \vert^{(1/p - 1/q)q'}\right)^{1/q'} \|u\|_q \|P\|_{\mathcal P(^m\ell_p^n)}\\
        &\leq&m e^{1+\frac{m-1}p} \left(\sum_{\bj\in \mathcal{J}(m-1,n)} {|\bj|}|u_\bj|^q\right)^{1/q}\left(\sum_{\bj\in \mathcal{J}(m-1,n)} \vert \bj \vert^{(1/p - 1/q)q'}\right)^{1/q'} \|u\|_q \|P\|_{\mathcal P(^m\ell_p^n)}\\
        &=&m e^{1+\frac{m-1}p} \left(\sum_{\bj\in \mathcal{J}(m-1,n)} \vert \bj \vert^{(1/p - 1/q)q'}\right)^{1/q'} \|u\|_q^m \|P\|_{\mathcal P(^m\ell_p^n)},
    \end{eqnarray*}
    which gives the desired inequality.
\end{proof}
The key to prove the lower bound is to obtain good bounds for the sum on the right hand side of the previous lemma. This will require some hard work.

We define for any $1 \le k \le m $ the $k$-bounded index set as
\[
\Lambda_k(m,n) = \lbrace \alpha \in \Lambda(m,n) : \alpha_i \le k \mbox{ for all } 1 \le i \le n \rbrace.
\]
Let $F$ be the bijective mapping connecting $\Lambda(m,n)$ and $\Jj(m,n)$ defined as
\begin{eqnarray*}
F:\Jj(m,n) & \rightarrow \Lambda(m,n) \\
\bj & \mapsto \alpha
\end{eqnarray*}
where $\alpha_i = F(\bj)_i = \#\lbrace k\,:\, \bj_k = i \rbrace$ for every $1 \le i \le n$. We denote
\[
\Jj_k(m,n) = F^{-1} (\Lambda_k(m,n)),
\]
for the corresponding $k$-bounded subsets of $\Jj(m,n)$. Observe that for any $1 \le k \le m$ and $\bj \in \Jj_k(m,n) $ the following hold:
\begin{equation}\label{card jj}
|\bj| \ge \frac{m!}{k!^{ \lceil \frac{m}{k} \rceil }} \ge \frac{m!}{k!^{ \frac{m}{k} + 1}},
\end{equation}
\begin{equation}\label{k split}
  \left(\sum_{\bj\in \mathcal{J}(m-1,n)} \vert \bj \vert^{(1/p - 1/q)q'}\right)^{1/q'} \le 2^{1/q'} \max \Big\lbrace \left(\sum_{\bj\in \mathcal{J}_k(m-1,n)} \vert \bj \vert^{(1/p - 1/q)q'}\right)^{1/q'} \; , \; \left(\sum_{\bj\in \mathcal{J}_k^c(m-1,n)} \vert \bj \vert^{(1/p - 1/q)q'}\right)^{1/q'} \Big\rbrace,
\end{equation}
%and,
%\begin{align}\label{card no acot}
%| \Jj^c_k(m,n) |
%	& \le n |\Jj(m-k-2, n)|  \\
% 	& \le n \frac{(n + m - k - 1)^{m-k-1}}{(m-k-1)!} , \nonumber
%\end{align}
and finally,
\begin{align}\label{card no acot}
| \Jj^c_k(m-1,n) |
	& \le n |\Jj(m-k-2, n)|   \\
 	& \le n\binom{n+m-k-3}{m-k-2} \le n \frac{(n + m - k - 3)^{m-k-2}}{(m-k-2)!}, \nonumber
\end{align}
since $\bj \in \Jj^c_k(m-1,n)$ requires that at least one of the variables is at the power of $k+1$. For the particular case $ m \le n$ we can extract from inequality \eqref{card no acot} the fact that
\begin{equation}\label{card no acot m < n}
| \Jj^c_k(m-1,n) | \le 2^m \frac{n^{m-k-1}}{(m-k-2)!}.
\end{equation}
Note also that,
\begin{equation}\label{km split}
\left(\sum_{\bj\in \mathcal{J}(m-1,n)} \vert \bj \vert^{(1/p - 1/q)q'}\right)^{1/q'} \le m^{1/q'} \max_{k=1,\dots,m-1} \Big\lbrace \left(\sum_{\bj\in \mathcal{J}_k(m-1,n)\cap \mathcal{J}_{k-1}^c(m-1,n)} \vert \bj \vert^{(1/p - 1/q)q'}\right)^{1/q'}  \Big\rbrace,
\end{equation}

\medskip

\begin{lemma}\label{cota jj 1}
For $1 < q \le p \le 2$ and for $m,n \in \NN$ fulfilling $ m \ge \log(n)^{\frac{q'}{p'}} $, it follows
\[
\left( \sum_{\bj \in \mathcal{J}(m-1,n)} \vert \bj \vert^{(1/p - 1/q)q'} \right)^{1/q'}  \le C^m  \frac{n^{m/q'}}{\log(n)^{m/p'}}.
\]
\end{lemma}

\begin{proof}
For $ m \ge \log(n)^{\frac{q'}{p'}}$ we just bound $ \vert \bj \vert^{(1/p - 1/q)q'}$ by 1, thus we have by Stirling formula,
\begin{align*}\label{cuenta}
 \left(\sum_{\bj\in \mathcal{J}(m-1,n)} \vert \bj \vert^{(1/p - 1/q)q'}\right)^{1/q'}
    & \le  |\mathcal{J}(m-1,n)|^{1/q'}  \\
    & =  \left( \frac{(n+m-2)!}{(m-1)! (n-1)!} \right)^{1/q'}  \\
    & \le  \left( c_1^{m-1} \left( 1 + \frac{n}{m-1} \right)^{m-1} \right)^{1/q'}  \\
    & \le  C_1^\frac{m-1}{q'} \left( 1 + \frac{n}{\log(n)^{\frac{q'}{p'}}} \right)^\frac{m-1}{q'}   \\
    & \le  C_2^m \frac{n^\frac{m-1}{q'}}{\log(n)^{\frac{m-1}{p'}}} \\
    & \le  C^m \frac{n^\frac{m}{q'}}{\log(n)^{\frac{m}{p'}}}.
\end{align*}
\end{proof}

\begin{lemma}\label{cota jj 2}
For $1 < q \le p \le 2$ and for $m,n \in \NN$ fulfilling  $m \le \frac{\log(n)}{\log\log(n) \beta}$ with $\beta = q' \left( \frac{1}{q} - \frac{1}{p} \right)$ it follows
\[
\left( \sum_{\bj \in \mathcal{J}(m-1,n)} \vert \bj \vert^{(1/p - 1/q)q'} \right)^{1/q'}  \le C^m\frac{n^{m/q'}}{\log(n)^{m/p'}}.
\]
\end{lemma}

\begin{proof}

Now let $m \le \frac{\log(n)}{\log\log(n) \beta}$,  we will use inequality \eqref{k split} for $k = 1$. First, being $k = 1$, we have
\begin{align*}
  \left( \sum_{\bj \in \mathcal{J}_1(m-1,n)} \vert \bj \vert^{(1/p - 1/q)q'} \right)^{1/q'}
	& = \frac{1}{m!^{1/q - 1/p}} |\mathcal{J}_1(m-1,n)|^{\frac{1}{q'}} \\
	& \le C^m \frac{n^{(m-1)/q'}}{m^{m/p'}}.
\end{align*}
On the other hand
\begin{align*}
\left(\sum_{\bj\in \mathcal{J}_1^c(m-1,n)} \vert \bj \vert^{(1/p - 1/q)q'}\right)^{1/q'}
	& \le |\mathcal{J}_1^c(m-1,n)|^{1/q'} \\
	& \le C^m \left( \frac{n^{m-2}}{m-2!} \right)^{1/q'} \;\; \mbox{using inequality \eqref{card no acot m < n}} \\
	& \le C^m \left( \frac{n^{m-2}}{m!} \right)^{1/q'}.
\end{align*}
Now its enough to prove the bound
\[
\left( \frac{n^{m-2}}{m!} \right)^{1/q'} \le C^m \frac{n^{(m-1)/q'}}{m^{m/p'}},
\]
which is equivalent to
\[
m^{(1/q - 1/p)q'} \le C n^{1/m} ,
\]
as $m \le \frac{\log(n)}{\log\log(n) \beta}$ we have, for some $C > 0$
\begin{align*}
m^{(1/q - 1/p)q'} & \le \left( \frac{\log(n)}{\log\log(n) \beta} \right)^\beta  \le C \log(n)^\beta = C e^{\log\log(n) \beta} \\	
	& = C e^\frac{\log(n)\log\log(n) \beta}{\log(n)}  = C n^{\frac{\log\log(n) \beta}{\log(n)}} \le C n^{1/m}.
\end{align*}
Therefore we have, for some $C>0$,
\[
\left( \sum_{\bj \in \mathcal{J}(m-1,n)} \vert \bj \vert^{(1/p - 1/q)q'} \right)^{1/q'} \le C^m \frac{n^{(m-1)/q'}}{m^{m/p'}}.
\]
To finish the proof, note that using Remark \ref{opt f(m)},
$$
\frac{n^{(m-1)/q'}}{m^{m/p'}}=\Big[\frac{n^{1/q'}}{m^{1/p'}n^{m/q'}}\Big]^m\le C^m\frac{n^{m/q'}}{\log(n)^{m/p'}}.
$$
\end{proof}

\begin{lemma}\label{cota jj}
	For $1 < q \le p \le 2$ and for $m,n \in \NN$ fulfilling $\log(n)^\frac1{c}\le m \le \log(n)^c $, for some $c>1$. Then there exists $C>0$ such that,
	\[
	\left(\sum_{\bj\in \mathcal{J}(m-1,n)} \vert \bj \vert^{(1/p - 1/q)q'}\right)^{1/q'}   \le C^m  \frac{n^{m/q'}}{\log(n)^{m/p'}}.
	\]
\end{lemma}
\begin{proof}
	By \eqref{card jj} and Stirling formula, we have, for each $1\le k\le m-1$,
	\begin{eqnarray*}
		\left(\sum_{\bj\in\mathcal{J}_k(m-1,n)\cap \mathcal{J}_{k-1}^c(m-1,n)} \vert \bj \vert^{(1/p - 1/q)q'}\right)^{1/q'}
		& \le & C^{m}|\mathcal{J}_{k-1}^c(m-1,n)|^{1/q'} \frac{k^{(m+k)(\frac1{q}-\frac1{p})}}{m^{m(\frac1{q}-\frac1{p})}}\nonumber \\
		& \le & C^{m}\left(\frac{n^{m-k}}{(m-k-1)^{m-k-1}}\right)^{\frac1{q'}}\frac{k^{(m+k)(\frac1{q}-\frac1{p})}}{m^{m(\frac1{q}-\frac1{p})}}. \nonumber\\
	\end{eqnarray*}
	Thus, by \eqref{km split}, we will prove the lemma if we are able to show that this last expression is $\le C^m \frac{n^{m/q'}}{\log(n)^{m/p'}}$, for some constant $C>0$.
	Therefore, it suffices to prove that, if $\beta:=(\frac1{q}-\frac1{p}){q'}$,
	\begin{equation}\label{jj bvq}
	\frac{k^{\beta(m+k)}}{(m-k-1)^{m-k-1}m^{\beta m}}\le C^m  \frac{n^{k}}{\log(n)^{m(\beta+1)}}.
	\end{equation}
	Let us first suppose that $k\ge \min\{m/2,\frac{m}{3\beta}\}= dm$ for some $0<d<1$. Then, bounding $k$ by $m-1$, the left hand side is less than or equal to $m^{\beta m}$, which is $\le \frac{n^{dm}}{\log(n)^{m(\beta+1)}}$ for big enough $n$.
	Note that, for $k\le dm$, \eqref{jj bvq} is equivalent to
	\begin{equation}\label{jj bvq2}
	\frac{k^{\beta(m+k)}}{m^{m-k}m^{\beta m}}\le C^m  \frac{n^{k}}{\log(n)^{m(\beta+1)}},
	\end{equation}
	for some constant $C$.
	Thus, for $1<k\le dm$ (for $k=1$ \eqref{jj bvq2} is trivially satisfied), since $k^{\beta k}m^{k}/k^m\le m^{k}k^{m/3}/k^m=  m^k/k^{2m/3}\ll 1$, it is enough to show that
	\begin{equation*}
	\frac{k^{(\beta+1)m}}{m^{(\beta+1) m}}\le C^m  \frac{n^{k}}{\log(n)^{m(\beta+1)}},
	\end{equation*}
	or,
	\begin{equation}\label{jj bvq3}
	(\beta+1)m\log(\frac{k\log(n)}{m})-k\log(n)\le Cm.
	\end{equation}
	Note that this inequality holds trivially if $k\log(n)\le m$ or if $k\log(n)\ge (\beta+1)m\log(\frac{k\log(n)}{m})$. Suppose then that $1< \frac{k\log(n)}{m}< (\beta+1)\log(\frac{k\log(n)}{m})=f(\frac{k\log(n)}{m})$, where $f$ is the logarithm in base $e^{\frac1{\beta+1}}$. Thus by \eqref{jj bvq3}, it suffices to see that
	\begin{equation}\label{jj bvq4}
	m\cdot f^{\circ 2}(\frac{k\log(n)}{m})-k\log(n)\le Cm,
	\end{equation}
	%or equivalently
	%\begin{equation}
	%{(\beta+1)mf_2(\frac{k\log(n)}{m})}-k\log(n)\le Cm,
	%\end{equation}
	where, $f^{\circ j}$ denotes the function $f$ composed with itself $j$ times. Again, \eqref{jj bvq4} is true if $k\log(n)\ge mf^{\circ 2}(\frac{k\log(n)}{m})$, and if this does not hold, then replacing in \eqref{jj bvq3}, it suffices to see that
	\begin{equation}
	{mf^{\circ 3}(\frac{k\log(n)}{m})}-k\log(n)\le Cm.
	\end{equation}
	We can continue this process, and it is enough to prove that, for some $j$,
	\begin{equation}\label{jj bvq5}
	{mf^{\circ j}(\frac{k\log(n)}{m})}-k\log(n)\le Cm.
	\end{equation}
	Finally, note that for some $t_0=t_0(\beta)$, which may suppose is bigger than 2, $f(t)\le t^{1/2}$ for every $t\ge t_0$. Let $t=\frac{k\log(n)}{m}$. Then if $\min\{t,f(t),\dots,f^{\circ i}(t)\}\ge t_0$, we have
	$$
	f^{\circ (i+1)}(t)\le (f^{\circ i}(t))^{1/2}\le  (f^{\circ (i-1)}(t))^{1/4} \le \dots \le t^{1/2^{i+1}}.
	$$
	Therefore, for some $j$, we will have that $f^{\circ j}(t)<t_0$ and \eqref{jj bvq5} is fulfilled taking $C=t_0$.
%	Since, for any $t>1$, $f^{\circ j}(t)$ is either eventually $<0$ (if $\frac1{\beta+1}>\frac1{e}$) or converges to a fixed point of $f$ (for $0<\frac1{\beta+1}\le \frac1{e}$),  \eqref{jj bvq5} must hold for $j$ large enough. \newline
%	
%	CHEQUEAR ESTO!! CREO QUE VALE, VER SI SIRVE ESTO PARA JUSTIFICAR: Iterating the Logarithmic Function, Xianglong Ni. The College Mathematics Journal , Vol. 47, No. 3 (May 2016).
\end{proof}

\begin{proof}[Proof of the lower bound of the case $1 < q \le p \le 2$ on Theorem \ref{maintheorem}.]
Thanks to Lemma \ref{Bohr vs unc} it is enough to prove that
\begin{equation}\label{lower bound p <2}
% \frac{ \log(n)^{1-1/p} }{n^{1 - 1/q} \log\log(n)^{1/q - 1/p}} \ll \frac{1}{ \sup_{ m \ge 1}  \chi_M(\Pp(^m {\ell_p^n}),\Pp(^m {\ell_q^n}))^{1/m}} = \inf_{m \ge 1} \frac{1}{\chi_M(\Pp(^m {\ell_p^n}),\Pp(^m {\ell_q^n}))^{1/m}},
\frac{ \log(n)^{1-1/p} }{n^{1 - 1/q}}  \ll \frac{1}{ \sup_{ m \ge 1}  \chi_M(\Pp(^m {\ell_p^n}),\Pp(^m {\ell_q^n}))^{1/m}} = \inf_{m \ge 1} \frac{1}{\chi_M(\Pp(^m {\ell_p^n}),\Pp(^m {\ell_q^n}))^{1/m}},
\end{equation}
which follows by Lemma's \ref{cteincondicionalidad},   \ref{cota jj 1}, \ref{cota jj 2} and \ref{cota jj}.
%, for some $C = C(p,q)>0$ it holds
%\begin{equation}\label{m hom upper bound p < 2}
%\chi_M(\Pp(^m {\ell_p^n}),\Pp(^m {\ell_q^n}))^{1/m} \le C \frac{n^{1/q'}}{\log(n)^{1/p'}},
%\end{equation}
%as we wanted to prove.
\end{proof}

\subsection{The case $ p \ge 2$}

For the remaining cases it will be crucial the monomial convergence point of view from \cite{bayart2017multipliers, defant2009domains}.

As a consequence of \cite[Example 4.9 (2)]{defant2009domains} we have that whenever $ \frac{1}{r} = \frac{1}{p} + \frac{1}{2}$ and $p \ge 2$ it follows
\[
\ell_r \cap B_{\ell_p} \subset \dom H_{\infty}(B_{\ell_p}).
\]
Since $ r \le p $,  then $B_{\ell_r} \subset B_{\ell_p}$ and then
$B_{\ell_r}\subset \ell_r \cap B_{\ell_p} \subset   \dom H_{\infty}(B_{\ell_p})$. Finally by Lemma \ref{general conv mon} we have that there is some constant $C = C(p,r) > 0$ such that for every $n \in \NN$ and $p,r$ fulfilling the previous conditions it holds
\[
\frac{1}{\dis\sup_{m \ge 1}\Big(\chi_M(\Pp(^m {\ell_p^n}),\Pp(^m {\ell_r^n}))\Big)^{1/m}} \ge C.
\]
As $K(B_{\ell_p^n}, B_{\ell_r^n}) \le 1$ for every $ 1 \le p,r \le \infty$, the previous inequality and Theorem \ref{Bohr vs unc} lead us to the assertion that for $ \frac{1}{r} = \frac{1}{p} + \frac{1}{2}$
\begin{equation}\label{1/q = 1/p + 1/2}
K(B_{\ell_p^n}, B_{\ell_r^n}) \sim 1.
\end{equation}

\begin{proof}[Proof of the case $\frac{1}{2} + \frac{1}{p} \le \frac{1}{q}$ on Theorem \ref{maintheorem}.]
Let $p,q$ be such that  $\frac{1}{r} := \frac{1}{2} + \frac{1}{p} \le \frac{1}{q}$, then for any $f(z) = \dis\sum_{m \ge 0}\sum_{ \alpha \in \Lambda(m,n) } a_{\alpha} z^\alpha $, it follows that
\begin{align*}
\dis\sup_{z\in K(B_{\ell_p^n}, B_{\ell_r^n})B_{\ell_q^n}} \sum_{m \ge 0}\sum_{\alpha \in \Lambda(m,n)} |a_\alpha z^\alpha | & \le \dis\sup_{z\in K(B_{\ell_p^n}, B_{\ell_r^n})B_{\ell_r^n}} \sum_{m \ge 0}\sum_{\alpha \in \Lambda(m,n)} |a_\alpha z^\alpha | \\
& \le   \| f\|_{H_\infty(B_{\ell_p^n})} ,\\
\end{align*}
which implies that $ K(B_{\ell_p^n},B_{\ell_q^n})  \ge K(B_{\ell_p^n}, B_{\ell_r^n})$. Therefore by equation \eqref{1/q = 1/p + 1/2} we  have
\[
K(B_{\ell_p^n},B_{\ell_q^n})  \sim 1.
\]
%Again, as $K(B_{\ell_p^n},B_{\ell_q^n})  \le 1$ for every $ 1 \le p,q \le \infty$ previous inequality leads us to the assertion  that for $ \frac{1}{p} + \frac{1}{2} \le \frac{1}{q} $.
\end{proof}
For every $z \in \ell_\infty$ we can define $z^* \in \ell_{\infty}$ the decreasing rearrangement such that $ z_n^* \ge z_{n+1}^*$ for every $n \in \NN$. In \cite[Theorem 2.2]{bayart2017multipliers} the authors proved that
\[
 B_\infty := \left\lbrace z \in \ell_{\infty} : \limsup_{n \to \infty} \frac{1}{\sqrt{\log(n)}} \left( \dis\sum_{j =1}^n |z_j^*|^2 \right)^{1/2} < 1 \right\rbrace
	\subset \dom  H_{\infty}(B_{\ell_\infty}).
\]
Consider now the Banach sequence space
\[
X_{\infty} = \left\lbrace z \in \ell_{\infty} : \dis\sup_{n \ge 2} \frac{1}{\sqrt{\log(n)}} \left( \dis\sum_{j=1}^n |z_j^*|^2 \right)^{1/2} < \infty \right\rbrace,
\]
endowed with the norm $ \| z \|_{X_\infty} = \dis\sup_{n \ge 2} \frac{1}{\sqrt{\log(n)}} \left( \dis\sum_{j=1}^n |z_j^*|^2 \right)^{1/2}$, observe that
\begin{equation}\label{B dentro de mon}
B_{X_\infty}
	\subset B_\infty
	\subset \dom H_{\infty}(B_{\ell_\infty}).
\end{equation}

By Theorem \ref{general conv mon} and expression \eqref{B dentro de mon} we have for some $C = C(p)>0$
\begin{equation}\label{incond X infinito}
\sup_{n \ge 2} \chi_M(\Pp(^m (X_\infty)_n), \Pp(^m \ell_p^n)) \le C^m.
\end{equation}
Observe that the norm in $(X_\infty)_n$ coincides with the given by
\[
\| (z_1, \ldots, z_n) \|_{(X_\infty)_n} = \dis\sup_{2 \le k \le n} \frac{1}{\sqrt{\log(k)}} \left( \dis\sum_{j=1}^k |z_j^*|^2 \right)^{1/2} .
\]
For any $ q \ge 2$ and $z \in \CC^n$ it holds
\[
\left( \dis\sum_{j=1}^k |z_j^*|^2 \right)^{1/2} = \| (z_1^*, \ldots, z_k^*) \|_2 \le k^{\frac{1}{2} - \frac{1}{q}} \|  (z_1^*, \ldots, z_k^*) \|_q \le k^{\frac{1}{2}  - \frac{1}{q}} \| z \|_q,
\]
then  $ \| z \|_{(X_\infty)_n} \le \dis\sup_{2 \le k \le n} \frac{k^{\frac{1}{2} - \frac{1}{q}}}{\sqrt{\log(k)}} \| z \|_q $. As $0 \le \frac{1}{2} - \frac{1}{q} $ then $\frac{n^{\frac{1}{2} - \frac{1}{q}}}{\sqrt{\log(n)}} \underset{n \to \infty}{\rightarrow} \infty$, and there is some $C = C(q)$ such that $\frac{k^{\frac{1}{2} - \frac{1}{q}}}{\sqrt{\log(k)}} \le C \frac{n^{\frac{1}{2} - \frac{1}{q}}}{\sqrt{\log(n)}}$ for every $2\le k \le n$. Then we have
\begin{equation}\label{norm id q >2}
\| id : \ell_q^n \to (X_\infty)_n \| \le C \frac{n^{\frac{1}{2} - \frac{1}{q}}}{\sqrt{\log(n)}}.
\end{equation}

\begin{proof}[Proof  of the case $ 2 \le q,p$ on Theorem \ref{maintheorem}.]
%\begin{lemma}\label{radio de B inf vs q>2}
%For every $n \in \NN $ and $ 2 \le q,p$ it holds
%\[
%K(B_{\ell_p^n}, B_{\ell_q^n}) \gg \frac{\sqrt{\log(n)}}{n^{\frac{1}{2} + \frac{1}{p} -\frac{1}{q}}}.
%\]
%\end{lemma}

Thanks to the the fact that $B_{X_\infty} \subset \mon H_{\infty}(B_{\ell_\infty})$ and Theorem \ref{general conv mon} we have that there is some $C>0$ such that for every polynomial $P(z)= \displaystyle\sum_{ \alpha \in \Lambda(m,n)} a_\alpha z^\alpha$ and for every $z \in (X_\infty)_n$,
\[
\dis\sum_{\alpha \in \Lambda(m,n)} | a_\alpha| |z^\alpha| \le C^m \| z \|_{(X_\infty)_n}^m \| P \|_{\Pp(^m \ell_\infty^n)}.
\]
Thus for every $z \in \CC^n$ using \eqref{norm id q >2} we have
\[
\dis\sum_{\alpha \in \Lambda(m,n)} | a_\alpha| |z^\alpha| \le C^m \| z \|_{\ell_q^n}^m \left( \frac{1}{\sqrt{\log(n)}} n^{\frac{1}{2} - \frac{1}{q}} \right)^m \| P \|_{\Pp(^m \ell_\infty^n)}.
\]
For any $2\le p<\infty$, we have that $ \| P \|_{\Pp(^m \ell_\infty^n)} \le n^{m/p}  \| P \|_{\Pp(^m \ell_p^n)}$, and then
\[
\dis\sum_{\alpha \in \Lambda(m,n)} | a_\alpha| |z^\alpha| \le C^m\| z \|_{\ell_q^n}^m \left( \frac{1}{\sqrt{\log(n)}} n^{\frac{1}{2} + \frac{1}{p}- \frac{1}{q}} \right)^m \| P \|_{\Pp(^m \ell_p^n)},
\]
which implies
\[ \chi_M(\Pp(^m {\ell_p^n}),\Pp(^m {\ell_q^n}))^{1/m} \ll \frac{n^{\frac{1}{2} +\frac{1}{p} - \frac{1}{q}}}{\sqrt{\log(n)}}.
\]
Therefore, by Lemma \ref{Bohr vs unc}
\[
K(B_{\ell_p^n},B_{\ell_q^n})  \gg \frac{\sqrt{\log(n)}}{n^{\frac{1}{2}+ \frac{1}{p} -\frac{1}{q}}},
\]
as we wanted to prove.
\end{proof}

To complete the study of the mixed Bohr radius for $p \ge 2$ it remains to understand the case $ \frac{1}{2} \le \frac{1}{q} \le \frac{1}{2} + \frac{1}{p}$.

\begin{remark}\label{rem reinhardt vs infty}
For every Reinhardt domain $\Rr \subset \CC^n$, if $P \in \Pp(^m \CC^n) $ and $w \in \Rr$ if we define $P_w \in \Pp(^m \CC^n)$ as $P_w(z) = P(w \cdot z)$, it follows
\[
\| P_w \|_{\Pp(^m \ell_\infty^n)} \le \dis\sup_{z \in \Rr} | P(z)|,
\]
and $ a_\alpha(P_w) = a_\alpha(P) w^\alpha $. %$ c_\jj(P_w) = c_\jj(P) w_\jj $
\end{remark}

\begin{proof}[Proof  of the case $\frac{1}{2} \le \frac{1}{q} < \frac{1}{2} + \frac{1}{p}$ and $p \ge 2$  on Theorem \ref{maintheorem}.]
%\begin{lemma}\label{1/2 < 1/q < 1/2 + 1/p}
%Given $\frac{1}{2} \le \frac{1}{q} \le \frac{1}{2} + \frac{1}{p}$ and $p \ge 2$ it holds
%\[
%K(B_{\ell_p^n},B_{\ell_q^n})  \gg \frac{\sqrt{\log(n)}}{n^{\frac{1}{2}+ \frac{1}{p} -\frac{1}{q}}}.
%\]
%\end{lemma}

Fix $m \in \NN$ and take  $P\in \Pp(^m \CC^n)$, $P(z) = \dis\sum_{\alpha \in \Lambda(m,n)} a_\alpha z^\alpha  $. By Lemma \ref{Bohr vs unc}, it suffices to show that there exists some $C(p,q) > 0$ such that for every $z \in B_{\ell_q^m}$ it holds
\[
\dis\sum_{ \alpha \in \Lambda(m,n)} |a_\alpha z^\alpha |  \le C(p,q)^m \left( \frac{n^{\frac{1}{2}+ \frac{1}{p} -\frac{1}{q}}}{\sqrt{\log(n)}} \right)^m \| P \|_{\Pp(^m \ell_p^n)}.
\]

Consider now $y = (z_1^{\frac{p}{p+2}}, \ldots, z_n^{\frac{p}{p+2}})$ and $w = (z_1^{\frac{2}{p+2}}, \ldots, z_n^{\frac{2}{p+2}})$. It is easy to see that $ z = y \cdot w = (y_1 w_1, \ldots, y_n w_n)$, and thus, by \eqref{incond X infinito} and Remark \ref{rem reinhardt vs infty}, we have
\begin{align*}
\dis\sum_{ \alpha \in \Lambda(m,n)} |a_\alpha z^\alpha |
	& = \dis\sum_{ \alpha \in \Lambda(m,n)} |a_\alpha w^\alpha y^\alpha | \\
	& \le C^m \| y \|_{(X_\infty)_n}^m \| P_w \|_{\Pp(^m \ell_{\infty}^n)} \\
	& \le C^m \| y \|_{(X_\infty)_n}^m \| w \|_{\ell_p^n}^m \| P \|_{\Pp(^m \ell_p^n)}.
\end{align*}
It remains to check that
\[
\| y \|_{(X_\infty)_n} \| w \|_{\ell_p^n} \le C(p,q) \frac{n^{\frac{1}{2}+ \frac{1}{p} -\frac{1}{q}}}{\sqrt{\log(n)}}.
\]
To start let $1 \le k \le n$ then
\begin{align*}
\| (y_1^*, \ldots, y_k^*) \|_{\ell_{2}^k}
	& = \|(z_1^*, \ldots, z_k^*) \|_{\ell_{\frac{2p}{p+2}}^k}^{\frac{p}{p+2}} \\
	& \le \left( \|(z_1^*, \ldots, z_k^*) \|_{\ell_q^k} k^{\frac{1}{p}+\frac{1}{2} - \frac{1}{q}} \right)^{\frac{p}{p+2}} \\
	& \le \| z \|_{\ell_q^n}^{\frac{p}{p+2}} \left(  k^{\frac{1}{p}+\frac{1}{2} - \frac{1}{q}} \right)^{\frac{p}{p+2}},
\end{align*}
so we have
\begin{align*}
\| y \|_{(X_\infty)_n}
	& = \dis\sup_{2 \le k \le n} \frac{1}{\sqrt{\log(k)}} \| (y_1^*, \ldots, y_k^*) \|_{\ell_{2}^k} \\
	& \le \dis\sup_{2 \le k \le n} \frac{1}{\sqrt{\log(k)}} \| z \|_{\ell_q^n}^{\frac{p}{p+2}} \left(  k^{\frac{1}{p}+\frac{1}{2} - \frac{1}{q}} \right)^{\frac{p}{p+2}} \\
	& \le \dis\sup_{2 \le k \le n} \| z \|_{\ell_q^n}^{\frac{p}{p+2}} C(p,q) \frac{1}{\sqrt{\log(n)}} n^{\left( \frac{1}{p}+\frac{1}{2} - \frac{1}{q} \right)\frac{p}{p+2} }  \\
	&  C(p,q) \| z \|_{\ell_q^n}^{\frac{p}{p+2}} \frac{1}{\sqrt{\log(n)}} n^{\frac{1}{2} - \frac{1}{q}\frac{p}{p+2}}.
\end{align*}
On the other hand,
\begin{align*}
\| w \|_{\ell_p^n}
	& = \| z \|_{\ell_{\frac{2p}{p+2}}^n}^{\frac{2}{2+p}} \\
	& \le \| z \|_{\ell_q^n}^{\frac{2}{2+p}} n^{\left( \frac{1}{2} + \frac{1}{p} - \frac{1}{q} \right)\frac{2}{2 + p}} \\
	& = \| z \|_{\ell_q^n}^{\frac{2}{2+p}} n^{\frac{1}{p} - \frac{1}{q} \frac{2}{2 + p}}.
\end{align*}
Finally,
\begin{align*}
\| y \|_{(X_\infty)_n} \| w \|_{\ell_p^n}
	& \le C(p, q) \| z \|_{\ell_q^n}^{\frac{p}{p+2}} \frac{1}{\sqrt{\log(n)}} n^{\frac{1}{2} - \frac{1}{q}\frac{p}{p+2}} \| z \|_{\ell_q^n}^{\frac{2}{2+p}} n^{\frac{1}{p} - \frac{1}{q} \frac{2}{2 + p}} \\
	& = C(p, q) \| z \|_{\ell_q^n} \frac{n^{\frac{1}{2} +\frac{1}{p} - \frac{1}{q}}}{\sqrt{\log(n)}},
\end{align*}
as we needed.
\end{proof}

%\section{Main results} \label{Mainresult}
%
%
%
%If $(a_{n})_{n}$ and $(b_{n})_{n}$ are two sequences of real numbers we will write $a_{n} \ll b_{n}$ if there
%exists a constant $C>0$ (independent of $n$) such that $a_{n} \leq C b_{n}$ for every $n$.
%We will write $a_{n} \sim b_{n}$ if $a_{n} \ll b_{n}$ and $b_{n} \ll a_{n}$.
%Recall that the number of $m$-homogeneous monomials in $n$ variables is $|\mathcal{J}(m,n)|=\binom {n + m -1} {m} \sim n^m$.
%
%For every $P \in \mathcal{P}(^m \mathbb{C}^n)$ there exists a unique symmetric $m$-linear form $T$ such that for every $x \in \mathbb{C}^n$, $P(x)= T(x, \overset{m}{\ldots}, x)$, see \cite{dineen1999complex}.
%
%We now state our main theorem.

\section*{Acknowledgements}
We thank our friend Pablo Sevilla-Peris for his enormous generosity, the encouragement provided and for the various conversations we had around this topic.

\newcommand{\etalchar}[1]{$^{#1}$}

\end{document}